\documentclass[12pt]{article}
\usepackage{amsmath}
\usepackage{amssymb}

\usepackage[dvips]{graphicx}
\usepackage[usenames,dvipsnames,svgnames]{xcolor}

\newtheorem{theorem}{Theorem}[section]
\newtheorem{definition}[theorem]{Definition}
\newtheorem{lemma}[theorem]{Lemma}

\newenvironment{proof}
{\noindent
{\bf Proof.}}
{\hfill $\square$\medskip

}

\newcommand{\R}{\mathbb{R}}

\newcommand{\Z}{\mathbb{Z}}
\newcommand{\N}{\mathbb{N}}
\newcommand{\C}{\mathbb{C}}
\newcommand{\calK}{{\mathcal K}}
\renewcommand{\Im}{\operatorname{Im}}
\renewcommand{\P}{\mathbb{P}}

\usepackage{hyperref}

\begin{document}
\title{Clothoid Fitting and Geometric Hermite Subdivision}

\author{Ulrich Reif and Andreas Weinmann}

\date{\today}
\maketitle

\begin{abstract}
We consider geometric Hermite subdivision for planar curves, i.e., iteratively refining an 
input polygon with additional
tangent or normal vector information sitting in the vertices.
The building block for the (nonlinear) subdivision schemes we propose is based 
on clothoidal averaging, 
i.e., averaging w.r.t.\ locally interpolating clothoids, which are curves of 
linear curvature. 
To define clothoidal averaging, we derive a new strategy to approximate 
Hermite interpolating clothoids.
We employ the proposed approach to define the geometric Hermite analogues of
the well-known Lane-Riesenfeld and four-point schemes.
We present results produced by the proposed schemes and discuss their features. 
In particular, we demonstrate that the proposed schemes yield visually
convincing curves.	
\end{abstract}

\section{Introduction}
Linear subdivision schemes are widely used in various areas such as geometric 
modeling, multiscale analysis and for solving PDEs,  
and are rather well studied; references are for instance 
\cite{cavaretta1991stationary,dyn2002subdivision,han2018framelets,
peters2008subdivision}.

In the last two decades also the interest in nonlinear subdivision schemes has 
significantly increased. One class of such schemes  
deals with scalar real valued data, but employs nonlinear 
averaging/prediction/interpolation techniques, e.g., 
\cite{donoho2000nonlinear,kuijt2002shape,cohen2003quasilinear,Goldman2009OnTS}, 
making the corresponding schemes nonlinear.
Motivation may be to get more robust estimators for processing data or to be
able to deal with discontinuities or nonuniform data. 
Another class of nonlinear schemes addresses  data which 
live in a nonlinear space, such as a Riemannian manifold or a Lie group, see 
for instance 
\cite{wallner2005convergence,xie2009smoothness,grohs2010general,
weinmann2010nonlinear,moosm}.
A third class considers data living in Euclidean space, typically of dimension 
$2$ or $3$, but the averaging rules are nonlinear to take account of their 
geometric characteristics. Such rules may be formulated in terms of angles, 
perpendicular bisectors, interpolating circles, and the like, rather than 
acting on each component separately, as linear schemes do.
References of such {\em geometric schemes} include 
\cite{aspert2003non, sabin2004circle, marinov2005geometrically, dyn2009four, 
cashman2013generalized}.

In contrast to linear schemes, for which a rather well established analysis is 
available, nonlinear schemes are less well understood, and there are ongoing 
efforts to devise new and to improve existing tools. However, due to the 
nonlinear nature and the diversity of the proposed schemes, not as general 
results as in the linear case can be expected. The investigation of a particular 
class of nonlinear schemes is likely to require an additional particular 
analysis component not covered by a general theory.
Papers providing an analysis framework for geometric curve subdivision are 
\cite{dyn2012geometric, ewald2015holder}.
The first reference derives
sufficient conditions for a convergent interpolatory planar subdivision scheme 
to produce tangent continuous limit curves. 
The second reference deals with subdivision schemes
which are geometric in the sense that they commute with similarities 
and derives a framework to establish 
$C^{1,\alpha}$- and $C^{2,\alpha}$-regularity of the generated limit curves.

In this paper, we present a family of {\em geometric Hermite subdivision 
schemes} for the generation of planar curves where the data to be 
refined are point-vector pairs, the latter serving as information on 
tangents or normals. 
Schemes refining point-vector pairs of that type were already 
suggested in \cite{chalmoviansky2007non, lipovetsky2016weighted}.
The basic idea of these two approaches 
is to locally fit circles to the data and then to sample new points from them,
but the specific methods are different. Also the strategies for 
determining new vectors are not the same so that the two schemes have 
significantly different shape properties.
In contrast, the approach proposed here relies on clothoids,
which are curves of linear curvature, rather than circles.
The top row of Figure~\ref{fig:compare} shows the refinement
of initial data consisting of two point-normal pairs by the so-called {\em 
clothoid average}, as described in this paper. The resulting curve is S-shaped 
and interpolates the initial data. Also the scheme of 
Chalmoviansky and J{\"u}ttler \cite{chalmoviansky2007non},
as shown in the middle row, interpolates the initial data, but has three 
turning points, leading to a less natural shape and a less optimal distribution of 
curvature. The scheme of Dyn and Lipovetsky \cite{lipovetsky2016weighted}, 
shown in the bottom row, produces 
a straight line (as it would for any choice of parallel normals).
This is likely not to address the designer's intent, 
and it also does not interpolate the normal directions specified at the endpoints. 
The good performance of our scheme is related to the 
fact that it is not based on the reproduction of circles, which can be seen as 
the geometric analogue of quadratic polynomials,
but on the reproduction of clothoids, which can be seen as the geometric 
analogue of cubic polynomials. Thus, it is able to mimic a much larger variety 
of shapes.

The idea of geometric planar spline fitting, i.e., fitting splines based on 
clothoids as building blocks can be traced through the literature for more than 
50 years  
\cite{birkhoff1964piecewise, stoer1982curve, mehlum1974nonlinear, 
meek1992clothoid, hoschek1992grundlagen,coope1993curve}. 
Further recent contributions are for instance
\cite{bertolazzi2018interpolating}
presenting an alternative to \cite{stoer1982curve} for the computation of a 
$C^2$ interpolating clothoid spline.
Hermite interpolation problems w.r.t.\ clothoids are for instance the topic of \cite{bertolazzi2015g1}. 
We also point out
\cite{meek2009two} where the authors employ so called e-curves as approximative substitutes for
clothoids in the context of Hermite interpolation. 

Maybe the two striking reason why clothoid splines and approximations 
of them are still a topic in the literature
are as follows: (i) the corresponding clothoid splines are rather expensive two 
compute; (ii) there are plenty of applications. Let us discuss these points in 
more details.
Concerning (ii),
for a long time clothoids have been used 
by route designers as transitional curves between
straight lines and circular arcs, and between circular arcs of different radii;
see for instance \cite{meek1990offset,baass1984use}.
Nowadays, they are further used in connection with path planning for autonomous 
vehicles,  
e.g., \cite{berkemeier2015clothoid, alia2015local,biral2012intersection}
in computer vision and image processing \cite{kimia2003euler,lun2015inpainting},
in curve editing for design purposes \cite{havemann2013curvature}
or representing hand-drawn strokes sketched by a user 
\cite{mccrae2009sketching,baran2010sketching}.
Concerning (i), many of the above applications are rather time critical and it 
is often important 
to have algorithms which are as fast as possible at a given (sometimes moderate) 
approximation quality. Solving clothoidal spline fitting problems up to high 
precision can be done rather fast \cite {stoer1982curve} but even this is 
sometimes too expensive or not needed. 
Instead, frequently faster strategies typically using some approximation are 
employed.
For instance, \cite{baran2010sketching} locally fits clothoids and arcs as 
primitives and then optimizes w.r.t. a certain graph to obtain a global fit.
The paper \cite{havemann2013curvature}
uses a 
variational approach iteratively inserting control points 
and optimizing them such as to generate a polyline with linear
discrete curvature (approximating the clothoid segment; for details see also 
\cite{schneider2000discrete}.)
We mention that this approach uses subdivision (explicitly the analogue of 
corner cutting) for clothoid blending.
Other approaches replace clothoids by approximating curve segments which are 
easier to handle, e.g., 
\cite{meek2004arc,meek2009two}.
We point out that the clothoidal subdivision schemes suggested here may be 
used as a computationally rather cheap alternative to clothoid splines and that 
they may be employed in the various applications discussed above.

The paper is organized as follows:
After presenting some basic facts about two-point Hermite interpolation with 
clothoids in the next section, we consider its approximate solution 
in Section~\ref{sec:ApproxSol}. The formula we propose is explicit, fast to 
evaluate, and yields a very small error for a large range 
of input data.
In Section~\ref{sec:SubDivCA},
this result is used to define the so-called {\em clothoid average} of a pair of 
points and corresponding tangent directions, which in turn serves as a 
building block for new families of {\em geometric Hermite subdivision 
schemes.} 
In particular, we obtain geometric Hermite analogues of the 
Lane-Riesenfeld schemes and the four-point scheme.
In Section~\ref{sec:Examples}, we present results produced by the proposed 
schemes and discuss their features to illustrate the potential of the new 
method. As a first theoretical result, Section~\ref{sec:Proof} establishes 
convergence and $G^1$-continuity of the Lane-Riesenfeld-type algorithm of degree 
$1$. Concluding remarks and an outlook are given in 
Section~\ref{sec:Conclusion}.

\begin{figure}
\centering
\includegraphics[width=\textwidth]{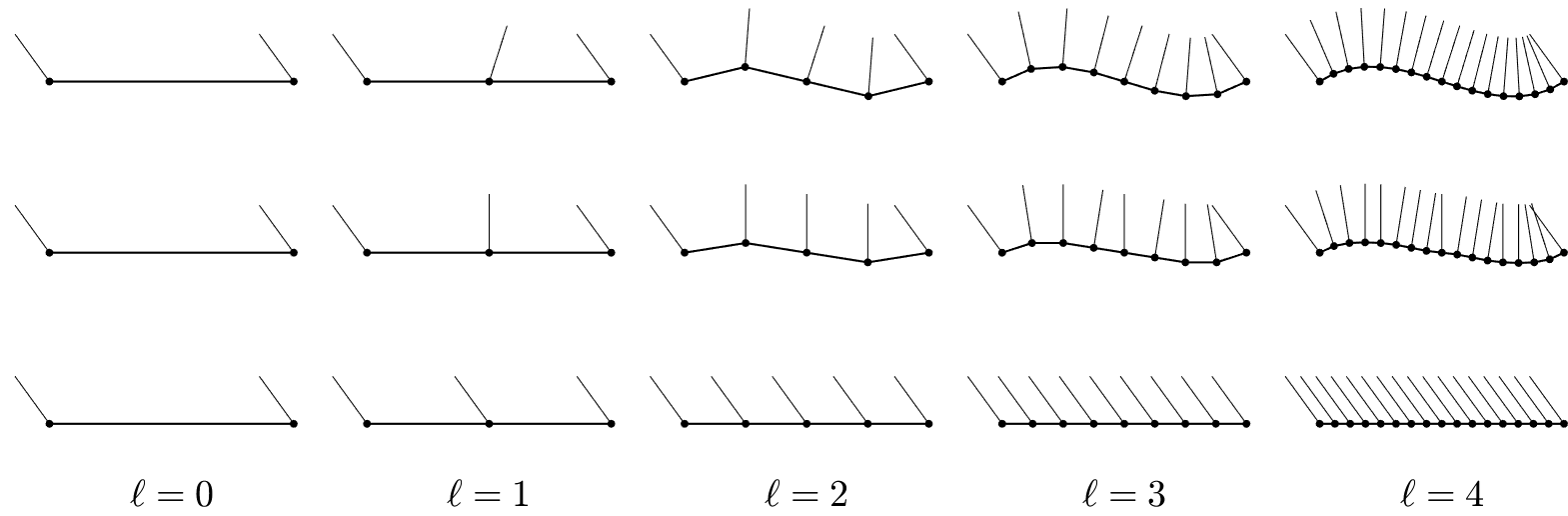}
\caption{Midpoint refinement using the proposed clothoid average {\em (upper row)},
the circle average of \cite{chalmoviansky2007non} {\em (middle row)}, and
the circle average of \cite{lipovetsky2016weighted} {\em (lower row)}. 
We always display the points $p_j^\ell$
together with the normals $n^\ell_j = i \exp(i\alpha^\ell_j)$ at level $\ell$.
}
\label{fig:compare}
\end{figure}

\section{Two-point interpolation}
\label{sec:TwoPInterpol}

In the following, points in the plane, and in particular images of planar 
curves, are considered as complex numbers. In this sense,
let $p : [0,1] \to \C$ be a
twice differentiable function parametrizing a planar curve which is 
{\em regular} in the sense that the {\em velocity} $v := |p'|$
vanishes nowhere. It is called {\em uniform} if the velocity is constant.
According to the lifting lemma, the tangent vector can be expressed 
in the form $p' = v \exp(i\alpha)$ with a differentiable  function
$\alpha : [0,1]\to \R$, called the {\em tangent angle} of $p$. 
We write $\alpha = \arg p'$ for brevity, and assume throughout that 
$\alpha$ is unrolled suitably so that jumps are avoided.
The {\em curvature} of $p$ is $\kappa := \alpha'/v = \Im \bar p' p''/v^3 $.
In the uniform case, on which we focus below,
curve, velocity, and tangent angle are related by the formula
\begin{equation}
\tag{1a}
\label{eq:p}
 p(t) = p(0) + v \int_0^t \exp(i\alpha(s))\, ds
 ,\quad 
 t \in [0,1]
 .
\end{equation}
\addtocounter{equation}{1}
The integral appearing here is abbreviated by 
\[
 I(\alpha,t) := \int_0^t \exp(i\alpha(s))\, ds
 ,\quad 
 I(\alpha) := I(\alpha,1)
 ,
\]
so that $p = p(0) + v I(\alpha,\cdot)$.
A curve $q : [0,1] \to \C$ starting at $q_0 := q(0) = 0$ and ending at $q_1 := 
q(1) \in \R^+$ is said to be in {\em normal position}. To tell this from the general case, 
we use the
letters $w = |q'|,\beta = \arg q'$, and $\lambda = \beta'/w$ to denote the
velocity, tangent angle, and curvature, respectively. 
We have the relation
\begin{equation}
\tag{1b}
\label{eq:q}
 q(t) = w I(\beta,t)
 ,\quad 
 t \in [0,1]
 .
\end{equation}
Curves in general and normal position are linked
by similarity. Denoting the secant between the endpoints of $p$ by 
$d := p_1-p_0$ and its angle by $\varphi := \arg d$, we have the relations
\begin{equation}
\label{eq:sim}
 p = p_0 + \frac{d}{q_1} q
 ,\quad 
 v = \frac{|d|}{q_1} w
 ,\quad
 \alpha = \beta + \varphi
 .
\end{equation}
Let $J = \{0,1\}$.
Given points $p_J = (p_0,p_1)$ with $d := p_1-p_0 \neq 0$ and angles 
$\alpha_J = (\alpha_0,\alpha_1)$, the corresponding {\em two-point
Hermite interpolation problem} is to find a curve $p$
such that 
\begin{equation}
 \label{eq:twopointG1}
 p(j) = p_j
 ,\quad
 \alpha(j) = \alpha_j
 ,\quad 
 j \in J
 .
\end{equation}
A special case of the above problem is to find a curve $q$ such that 
\begin{equation}
 \label{eq:twopointG1norm}
 q(j) = j
 ,\quad
 \beta(j) = \beta_j
 ,\quad 
 j \in J
 ,
\end{equation} 
for given $\beta_J$. We recall that we use letters $q,\beta$ to 
indicate that the sought curve is in normal position.
Figure~\ref{figure0} illustrates the setting.
While it is simple to specify curves merely satisfying these constraints, the 
challenge is to find solutions which are fair in some sense.
For instance, it is a classical task to solve \eqref{eq:twopointG1} 
in the space of clothoids, which are curves characterized by a linear curvature 
profile. In principle, this nonlinear problem is well understood and various 
more or less complicated methods for its numerical treatment are described in 
the literature, see for instance \cite{WALTON200986,bertolazzi2015g1} and the 
references therein.
Before we present our own approximate approach in the next section, we 
introduce some notation and basic facts.
\begin{figure}
\centering
\includegraphics[height=6cm]{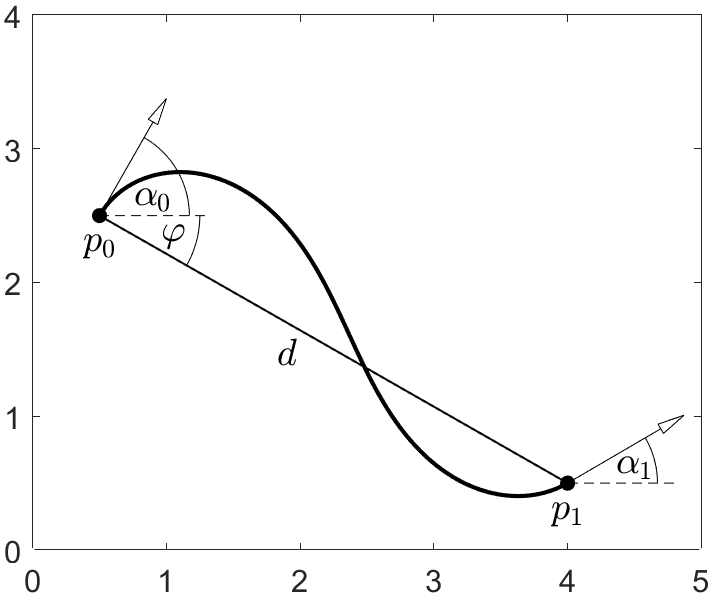}
\qquad 
\includegraphics[height=6cm]{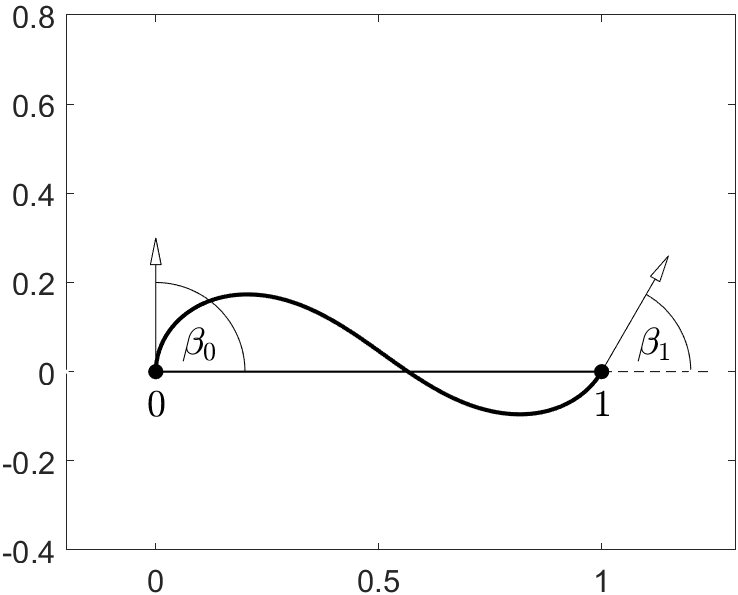}
\caption{Interpolation problem in general position {\it (left)} and 
normal position {\em (right)}.}
\label{figure0}
\end{figure}

Using the symbol $\P_n$ to denote the space of polynomials of degree at most $n$ over the 
unit interval, we define 
\[
 \calK_n := \{p: \kappa \in \P_n\}
\]
as the set of all uniform curves $p: [0,1] \to \C $ with curvature in 
$\P_n$.
The corresponding subset of such curves in normal position is denoted by $\calK_n^+$.
In particular, $\calK_0$  contains straight lines and circular arcs,
while curves in $\calK_1 \setminus \calK_0$ are segments of 
{\em clothoids}. 
The tangent angle $\alpha = \int \kappa$ of
a clothoid is a quadratic polynomial.
For clothoids,
the integral appearing in \eqref{eq:p} can be transformed to the so-called
{\em Fresnel integral} $F(x) := \int_0^x \exp(iu^2)\, du$, which 
does not possess a finite representation with respect to elementary functions.
For later use we state the following lemma.
\begin{lemma}
\label{lem:embedding}
A regular curve $p \in  \calK_1$ is embedded unless
it parametrizes a full circle, i.e., unless 
$p \in \calK_0$ and $|\alpha_1-\alpha_0| \ge 2\pi$.
\end{lemma}
The proof follows immediately form the Tait-Kneser Theorem.\\

In the following, for simplicity of wording, the term
{\em clothoid} addresses not only segments of true clothoids, but 
also segments of circles and straight lines. In this sense,
we consider the solution of \eqref{eq:twopointG1} with clothoids,
i.e., with curves $p\in \calK_1$.
By similarity according to \eqref{eq:sim},
this problem can be reduced to \eqref{eq:twopointG1norm} with data $\beta_J = 
\alpha_J-\arg d$, which are the angles between the secant $d$ and the
boundary tangent angles of $p$.
Using the quadratic Lagrange polynomials
\[
   \ell^2_0 := (t-1)(2t-1)
   ,\quad 
   \ell^2_{1/2}(t) := 4t(1-t)
   ,\quad 
   \ell^2_1(t) := t(2t-1)
   ,
\]
with respect to the break points $0,1/2,1$, the tangent angle of $q$ can be 
written as
\[
  \beta = \beta_0 \ell^2_0 + \beta_{1/2} \ell^2_{1/2} + \beta_1 \ell^2_1
.
\]
Hence, given $\beta_J$,
the sought solution is characterized by the value $\beta_{1/2} 
= \beta(1/2)$ and the velocity $w$ by means of \eqref{eq:q}. The task is 
to find these two values.
Figure~\ref{figure1} demonstrates that the solution of 
\eqref{eq:twopointG1norm} 
is not unique. More precisely, it is known that for each pair $\beta_J$ there 
exists a countable family of solutions, but in applications, one is typically 
interested in curves avoiding excess rotation. For instance, if the 
boundary data $\beta_J$ are small in modulus, also the overall maximum 
$\|\beta\|_\infty$ 
of tangent angles should be small. The following theorem guarantees existence 
of such a solution. 
\begin{figure}
\centering
\includegraphics[width=13cm]{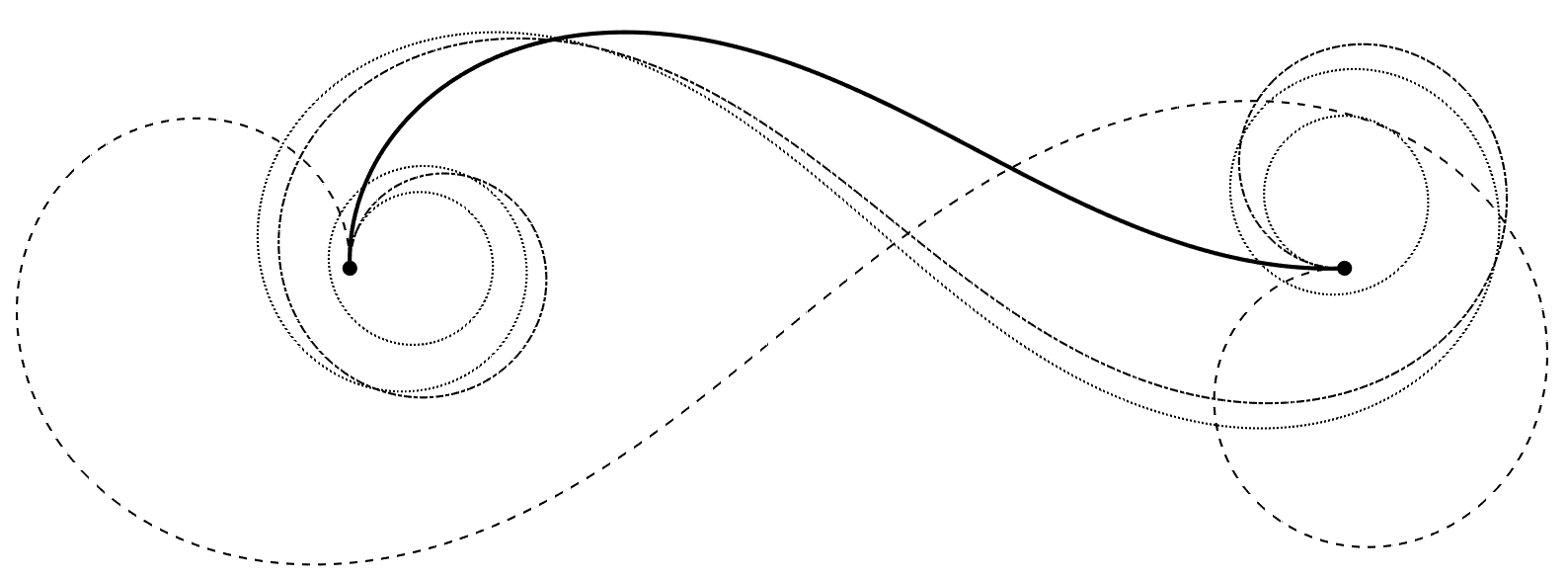}
\caption{Four out of infinitely many clothoid solutions to the geometric Hermite
interpolation problem \eqref{eq:twopointG1norm} with angles $\beta_0 = \pi/2, 
\beta_1 = 0$. 
The solid line shows the preferred choice, which avoids excess rotation.}
\label{figure1}
\end{figure}

\begin{theorem}
\label{thm:1}
There exists a smooth function $F : U \to \R^2$, defined on some 
neighborhood $U = (-u,u)^2$ of the origin, with 
\[
 F(0,0) =
 \begin{bmatrix}
  0\\1
 \end{bmatrix}
  ,\quad 
 DF(0,0) =
 \begin{bmatrix}
  -1/4 & -1/4 \\
  0 & 0
 \end{bmatrix}
  ,
\]
and the following property: Let
\[
 \begin{bmatrix}
 \beta_{1/2} \\ w
 \end{bmatrix}
 := F(\beta_0,\beta_1)
 ,\quad 
 \beta := \beta_0 \ell^2_0 + \beta_{1/2} \ell^2_{1/2} + \beta_1 \ell^2_1
 ,
\]
then the tangent angle $\beta$ and the velocity $w$ 
define a solution $q = w I(\beta,\cdot) \in \calK_1^+$ of 
\eqref{eq:twopointG1norm}. In particular, $I(\beta) = 1/w$ is real
and positive.
\end{theorem}
Actually, it is possible to choose the domain $U=(-\pi,\pi)^2$ for $F$, 
covering almost all possible pairs of boundary tangent angles, but this is not 
needed 
here.
\medskip 

\begin{proof}
The idea is to parametrize the set solutions of 
\eqref{eq:twopointG1norm} for varying $\beta_J$ as a two-dimensional 
surface in 
$\R^4$ and then to apply the implicit function theorem.
Given $\alpha_J \in \R^2$, define the tangent angle
$\alpha := \alpha_0 \ell^2_0 + \alpha_1 \ell^2_1$ and the clothoid 
$p := I(\alpha,\cdot)\in \calK_1$. 
Then $q := p/p(1) \in\calK_1^+$ connects $q(0) = 0$ and $q(1) = 1$.
Its velocity is $w = 1/|p(1)|$, and its tangent angle $\beta = \alpha - \arg 
p(1)$ has 
values $\beta_j := \beta(j) = \alpha_j- \arg p(1),j \in J,$ and 
$\beta_{1/2} := \beta(1/2) = - \arg p(1)$.
With these data, we define the surface
$\Phi(\alpha_0,\alpha_1) :=
[\beta_{1/2},w, \beta_0,\beta_1]^T$. It is well defined and smooth 
in a neighborhood of the origin since $p(1)=1$ for $\alpha_0=\alpha_1=0$.
Let $\|\alpha_J\|_\infty = h$. Then $\|\alpha\|_\infty = h$ and 
\[
  p(1) = 1 +\int_0^1 i\alpha(s)\, ds  + O(h^2) = 
  1 + i\,\frac{\alpha_0+\alpha_1}{6} + O(h^2)
\]
so that $w = 1 + O(h^2)$ and $\arg p(1) = (\alpha_0+\alpha_1)/6 + O(h^2)$.
We conclude that
\[
 \Phi(0,0) = 
 \begin{bmatrix}
  0\\1\\0\\0
 \end{bmatrix}
 ,\quad 
 D\Phi(0,0) = \frac{1}{6}
 \begin{bmatrix}
 -1 & -1 \\ 0 & 0 \\ 5 & -1 \\ -1 & 5
 \end{bmatrix}
 ,
\]
are value and derivative of $\Phi$ at the origin.
The lower $(2\times 2)$-submatrix of $D\Phi(0,0)$ has determinant $2/3$.
Hence, by the implicit function theorem, there exists a neighborhood
$U$ of the origin and a smooth function $F : U \to \R^2$ such that 
$[F(\beta_J),\beta_J]^T$ defines a point on the trace of $\Phi$ for 
all $\beta_J \in U$, corresponding to the clothoid solving 
\eqref{eq:twopointG1norm}. Value and derivative of $F$ at the origin 
are given by 
\[
 F(0,0) = 
 \begin{bmatrix}
 0 \\ 1 
 \end{bmatrix}
 ,\quad 
 DF(0,0) = 
 \begin{bmatrix}
  -1 & -1 \\ 0 & 0
 \end{bmatrix}
 \cdot
 \begin{bmatrix}
 5 & -1 \\ -1 & 5
 \end{bmatrix}^{-1}
 = 
 \begin{bmatrix}
 -1/4 & -1/4 \\ 0 & 0 
 \end{bmatrix}
  .
\]
One might suspect that employing functions $\alpha$ of the 
general form 
$\alpha = \alpha_0 \ell^2_0 + \alpha_{1/2} \ell^2_{1/2} + \alpha_1 \ell^2_1$ 
would define even more solutions of \eqref{eq:twopointG1norm}.
To see that this is not true, let
$\tilde \alpha := \alpha - \alpha_{1/2} = 
(\alpha_0-\alpha_{1/2})\ell^2_0 + (\alpha_1-\alpha_{1/2}) \ell^2_1$ be 
an angle function as considered above.
Then the corresponding clothoids $p := I(\alpha,\cdot)$ and 
$\tilde p := I(\tilde\alpha,\cdot)$ are related by
$p = \exp(i\alpha_{1/2}) \tilde p$ so that their normal forms
coincide, 
\[
 q = \frac{p}{p(1)} = 
 \frac{\exp(i\alpha_{1/2}) \tilde p}{\exp(i\alpha_{1/2}) \tilde p(1)} = 
 \frac{\tilde p}{\tilde p(1)} = \tilde q.
\]
\end{proof}
Let $p_J,\alpha_J$ be boundary data for the general problem 
\eqref{eq:twopointG1} with 
$d = p_1-p_0 \neq 0$.
If $q$ is the solution 
of \eqref{eq:twopointG1norm} for data $\beta_J := \alpha_J - \arg d$
according to the preceding theorem, then $p := p_0 + d q$
solves \eqref{eq:twopointG1}.
 With $q = w I(\beta,\cdot)$, an equivalent 
expression is
\begin{equation}
\label{eq:p2}
 p = p_0 + \frac{d}{I(\beta)} I(\beta,\cdot)
 ,
\end{equation}
which is independent of the velocity $w$. 
Hence, it suffices
to know the first coordinate function 
\[
 f : U \to \R
 ,\quad 
 f(\beta_0,\beta_1) = \beta_{1/2}
  ,
\]
of $F$ to construct the crucial tangent angle
$\beta :=\beta_0 \ell^2_0 + f(\beta_0,\beta_1) \ell^2_{1/2} + \beta_1 \ell^2_1$.

The function $f$ has the symmetry properties
\begin{equation}
\label{eq:symf}
 f(\beta_0,\beta_1) = f(\beta_1,\beta_0)
 ,\quad 
 f(-\beta_0,-\beta_1) = -f(\beta_0,\beta_1)
 .
\end{equation}
Hence, the second order derivatives vanish at the origin and we obtain 
the expansion
\[
 f(\beta_0,\beta_1) = 
 -\frac{\beta_0+\beta_1}{4} + O(\|\beta_J\|^3)
 .
\]


\section{Approximate solution}
\label{sec:ApproxSol}

Computing accurate solutions of the nonlinear problem \eqref{eq:twopointG1norm}  
is possible, but the determination of 
$\beta_{1/2} = f(\beta_0,\beta_1)$ requests more or less 
elaborate and/or computationally expensive numerical methods. 
Instead, we propose a good approximation
$\tilde f$ for later 
use with subdivision algorithms.
More precisely, we seek a function
$\tilde f$ with the following properties:
\begin{itemize}
\item[{\em i)}]
$\tilde f : \R^2 \to \R$ is a cubic polynomial. This choice combines modest 
complexity with sufficient flexibility to achieve good global approximation.
\item[{\em ii)}]
$\tilde f(0,0) = f(0,0)$ and $D\tilde f(0,0) = Df(0,0)$. Thus, $\tilde f$ 
approximates $f$ very good for small data. 
\item[{\em iii)}]
$\tilde f$ inherits the symmetry properties \eqref{eq:symf} from $f$,
\[
 \tilde f(\beta_0,\beta_1) = \tilde f(\beta_1,\beta_0)
 ,\quad 
 \tilde f(-\beta_0,-\beta_1) = -\tilde f(\beta_0,\beta_1)
 .
\]
\item[{\em iv)}]
For a wide range of boundary data, say 
$\beta_J \in B := [-\pi/2,\pi/2]^2$, the {\em angle defect} 
\[
  \delta (\beta_J) := \arg I(\beta)
  ,\quad 
 \beta := 
 \beta_0 \ell^2_0 + \tilde f(\beta_0,\beta_1) \ell^2_{1/2} + \beta_1 \ell^2_1
\]
is small in modulus. 
\end{itemize}
Before presenting our solution, let us explain the meaning of the angle defect
as a measure for the quality of the approximation $\tilde f$.
\begin{theorem}
Given Hermite data $p_J,\alpha_J$ with $d = p_1-p_0\neq 0$, let 
$\beta_J := \alpha_J - \arg d$ and $\beta := \beta_0 \ell^2_0 + \tilde 
f(\beta_0,\beta_1) \ell^2_{1/2} + \beta_1 \ell^2_1$, as above. 
Then, analogous to \eqref{eq:p2}, 
\[
 p := p_0 + \frac{d}{I(\beta)} I(\beta,\cdot)
\]
defines a clothoid interpolating the point data, i.e.,
$p(0) = p_0, p(1) = p_1$. At the boundaries, its tangent angle $\alpha$ differs
from the prescribed values by the angle defect,
\[
 \alpha(0) = \alpha_0 - \delta(\beta_J)
 ,\quad
 \alpha(1) = \alpha_1 - \delta(\beta_J)
 .
\]
\end{theorem}
\begin{proof}
Interpolation of the point data is trivial. For the tangent angle, we obtain
\[
 \alpha(j) = \arg d  + \beta(j) - \arg I(\beta)
 = \arg d + \beta_j - \delta(\beta_J) = \alpha_j - \delta(\beta_J)
 ,\quad 
 j \in J
 .
\]
\end{proof}
To construct a suitable function $\tilde f$, we adopt the idea used in the 
proof of Theorem~\ref{thm:1}. For a large collection of angles 
$\alpha_J^i, i\in I$, we compute points 
$\Phi(\alpha^i_0,\alpha^i_1) =[\beta^i_{1/2},w^i,\beta^i_0,\beta^i_1]^T$ 
representing clothoids in normal position. Points for  which one of the angles 
$\beta^i_0,\beta^i_{1/2},\beta^i_1$ lies outside the interval 
$B$ are discarded as they correspond to clothoids with too large tangent 
angles, which are of little relevance for most applications, e.g., for design purposes. Denoting the index 
set of the 
remaining points by $\tilde I$, the task is to determine $\tilde f$ such that 
$f(\beta^i_0,\beta^i_1) \approx \beta^i_{1/2}$ for all $i \in \tilde I$. The
ansatz for the cubic polynomial $\tilde f$ satisfying {\em ii)} with 
symmetry properties according to {\em iii)} is 
\[
 f(\beta_0,\beta_1) = (\beta_0+\beta_1) 
 \bigl(f_1 (\beta_0^2+\beta_1^2) + f_2 \beta_0 \beta_1 - 1/4\bigr)
 .
\]
Now, we can determine the unknown parameters $f_0,f_1$ by 
standard approximation methods. The following 
choice was found when striving for a good compromise between the maximal angle
defect $\|\delta(\beta_J)\|_\infty$ and 
simplicity of the coefficients.
\begin{theorem}
\label{thm:2}
Let 
\begin{equation}
\label{eq:tildef}
 \tilde f(\beta_0,\beta_1) := (\beta_0+\beta_1)\left(
 \frac{\beta_0^2+\beta_1^2}{68} - \frac{\beta_0 \beta_1}{46} - \frac{1}{4}
 \right)
 .
\end{equation} 
There exist constants $c_1,c_2$ such that 
the angle defect is bounded by
\[
  |\delta(\beta_0,\beta_1)| \le 
  \min\{c_1, c_2 |\beta_0+\beta_1|\cdot\|\beta_J\|^2\}
\]
for all $\beta_J \in B := [-\pi/2,\pi/2]^2$.
A feasible numerical value for both constants is $c_1 = c_2 = 1/800$, which is 
less than a $1/800$.
\end{theorem}
\begin{proof}
Boundedness of the angle defect by some constant $c_1$ follows immediately from 
continuity
over the compact domain $B$. Concerning existence of a bound in terms of 
$|\beta_0+\beta_1| \cdot \|\beta_J\|^2$, 
we note that properties {\em ii)} and {\em iii)} together with 
\eqref{eq:symf} imply that 
$f(\beta_J) - \tilde f(\beta_J) = O(\|\beta\|^3)$ and 
$f(\beta_J) = \tilde f(\beta_J) = 0$ for $\beta_0+\beta_1=0$.
Thus, there exists a constant $\tilde c_2$ such that 
\[
 |f(\beta_J) - \tilde f(\beta_J)| \le \tilde c_2 
|\beta_0+\beta_1|(\beta_0^2+\beta_1^2)
 ,\quad 
 \beta_J \in U
 .
\]
It remains to show that this qualitative behavior is inherited by $\delta$.
To this end, we define the function 
\[
 \Delta(\beta_J,\beta_{1/2}) := \bigl |\arg I(\beta)\bigr|
 ,\quad 
 \beta_J \in B
 ,\
 \beta_{1/2} \in f(B) \cup \tilde f(B)
 ,
\]
with $\beta = \beta_0 \ell^2_0 + \beta_1 \ell^2_1 + \beta_{1/2} \ell^{2}_{1/2}$.
Since $|\beta_1-\beta_0| < 2\pi$ for $\beta_J \in B$, 
Lemma~\ref{lem:embedding} implies that the integral is nonzero so that 
$\Delta$ is a smooth function over a compact domain and hence Lipschitz
with some constant $L$. We obtain 
\begin{align*}
 \delta(\beta_J) &= \Delta(\beta_J,\tilde f(\beta_J))
 = 
 \Delta(\beta_J,\tilde f(\beta_J)) - \Delta(\beta_J,f(\beta_J))\\
 &\le L |\tilde f(\beta_J) - f(\beta_J)|
 \le \tilde c_2 L|\beta_0+\beta_1|(\beta_0^2+\beta_1^2)
\end{align*}
for $\beta_J \in U$. For $\beta_J \not\in U$, continuity of $\delta$ 
shows that the inequality remains valid with a possibly enlarged constant
$c_2$.
The given numerical value for $c_1$ and $c_2$ can be verified by evaluation
over a fine grid on $B$, see Figure~\ref{figure2}.
\end{proof}
\begin{figure}
 \centering
 \includegraphics[height=5.5cm]{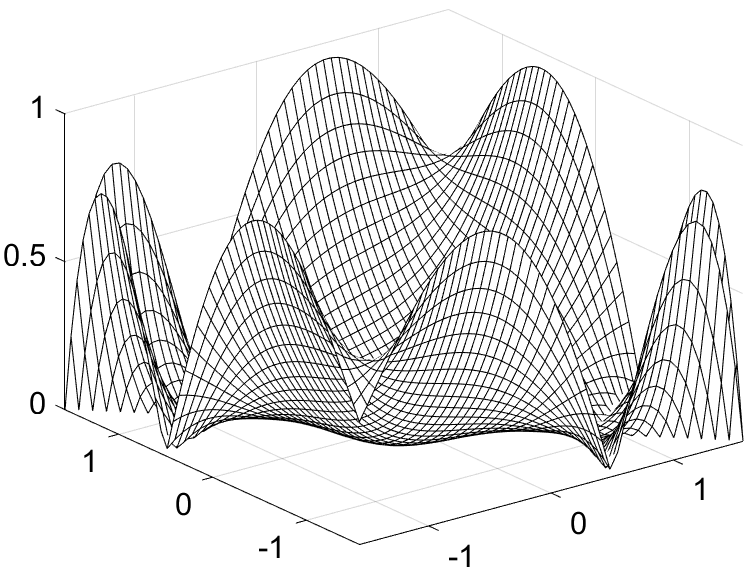}
 \qquad 
 \includegraphics[height=5.5cm]{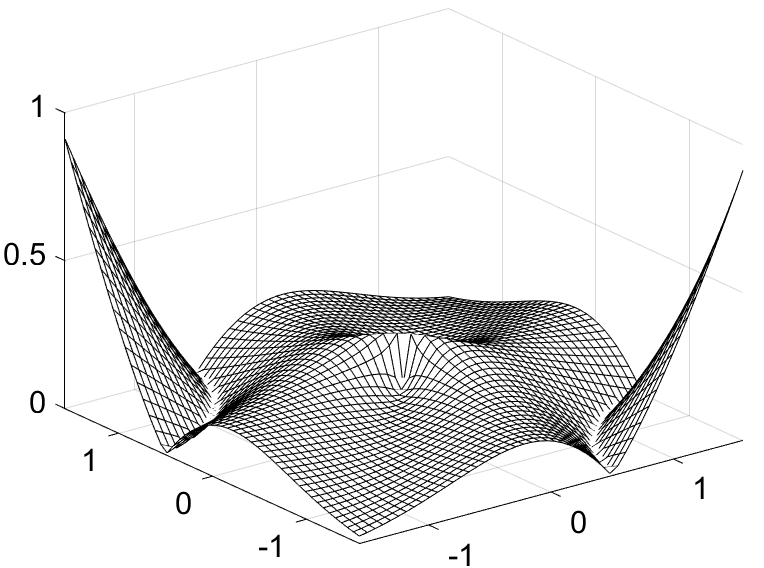}
 \caption{Plots of $800\cdot|\delta(\beta_J)|$ {\em (left)} and
 $800\cdot|\delta(\beta_J)|/(|\beta_0+\beta_1|(\beta_0^2+\beta_1^2))$ 
 {\em (right)}
 for $-\pi/2 \le \beta_0,\beta_1 \le \pi/2$.
 The upper bound $1$ in both cases indicates that $c_1=c_2=1/800$ is a 
 feasible value for 
 the constants in Theorem~\ref{thm:2}.}
 \label{figure2}
\end{figure}
In many applications, an error of less than a tenth of a degree for the 
interpolation of tangent angles will be 
acceptable so that the given approximation can be used directly.
Moreover, the theorem shows that the approximation is exact for the 
symmetric case 
$f(\beta_0,-\beta_0) = \tilde f(\beta_0,-\beta_0) = 0$, whose solution is a 
segment of a circle or a straight line.

If, in the general case, higher accuracy is required, one step of the Newton 
iteration

\[
\beta_{1/2}
\quad {\mathrel{\reflectbox{\ensuremath{\longmapsto}}}}\quad
\beta_{1/2} - \arg I(\beta) /
\operatorname{Re} \left(
\frac{\int_0^1 \ell^2_{1/2}(t) \exp(i\beta(t))\,dt}{I(\beta)}
\right)
\]

reduces the maximal angle defect to less than $5\!\times\!10^{-8}$, and a 
second step to less than $5\!\times\!10^{-16}$.


\section{Hermite subdivision by the clothoid average}
\label{sec:SubDivCA}

If a whole sequence of Hermite data points is to be interpolated, 
one might join always two of them by a clothoid and connect the
segments to form a single composite curve. If the clothoids are 
determined exactly, their contact is $G^1$, meaning that points and
tangent directions of neighboring clothoids coincide at the junctions.
If the clothoid is only approximated, as described in the preceding section, 
the contact is still continuous, but not $G^1$ in a strict sense.
In any case, curvature is discontinuous, what may be insufficient for 
many design applications. Below, we propose different subdivision strategies
to generate a (visually) smooth curve from a sequence
$p_j^0$ of points in $\C$ and corresponding tangent angles 
$\alpha_j^0, j \in\Z$. As building 
block we define the (approximate) clothoid average as follows:
\begin{definition}
Given a pair of {\em Hermite couples} $h_j := (p_j,\alpha_j), j \in J$,
with $d:= p_1-p_0 \neq 0$,
let $p \in \calK_1$ be the clothoid with endpoints $\tilde p(j) = p_j$ 
and tangent angles $\alpha(j) \approx \alpha_j$ constructed 
by the approximation $\tilde F$ of $F$ according to Theorem~\ref{thm:2}. Then 
the approximate {\em clothoid average} of $h_0$ and $h_1$ at $t \in \R$ is 
defined by evaluation of $p$ and the corresponding angle function $\alpha$ at 
$t$, and  written as
\[
 t h_0 \oplus (1-t) h_1 := \bigl(p(t), \alpha(t)\bigr)
 .
\]
\end{definition}
Sequences of Hermite couples are denoted by $H := (h_j)_{j \in \Z}$.
Now, we generate sequences $H^0,H^1,H^2,\dots$ from given initial 
data $H^0$ by means of a binary subdivision operator $S : H^\ell \to 
H^{\ell+1}$, the rules of which are based on the clothoid average.

The simplest subdivision operator of that type is inserting a new 
Hermite couple always between two given ones,
\[
 S_1 := H \mapsto H'
 ,\quad 
 h'_{2j} = h_j
 ,\quad 
 h'_{2j+1} = \frac{1}{2} h_j \oplus \frac{1}{2} h_{j+1}
 .
\]

Formally, $S_1$ corresponds to the Lane-Riesenfeld algorithm of degree $1$.
Defining the {\em averaging operator}
\[
 A := H \mapsto H'
 ,\quad 
 h'_{j} =  \frac{1}{2} h_j \oplus \frac{1}{2} h_{j+1}
 ,
\]
we can proceed in that direction and define 
Lane-Riesenfeld-type algorithms of degree $n$ by 
applying $(n-1)$ rounds of averaging to the output of $S_1$,
\[
 S_n := A^{n-1} S_1
 ,\quad 
 n \in \N
 .
\]
For instance, the Chaikin-type algorithm, obtained for $n=2$, explicitly reads
\[
 S_2 := H \mapsto H'
 ,\quad 
 h'_{2j} = \frac{1}{2}h_j \oplus 
 \frac{1}{2} \Bigl(\frac{1}{2} h_j \oplus \frac{1}{2} h_{j+1}\Bigr)
 ,\quad 
 h'_{2j+1} = 
 \frac{1}{2} \Bigl(\frac{1}{2} h_j \oplus \frac{1}{2} h_{j+1}\Bigr) \oplus
 \frac{1}{2}h_{j+1}
 .
\]
Needless to say that this symbolic representation of a nonlinear process
cannot be simplified through commutativity, associativity, or distributivity.

As a further example, we define a family of interpolatory four-point schemes
$S^4_\omega$ with tension parameter $\omega < 0$,
\[
 S^4_\omega := H \mapsto H'
 ,\quad 
 h'_{2j} := h_j
 ,\quad
 h'_{2j+1} := 
 \frac{1}{2} 
 \bigl(\omega h_{j-1} \oplus (1-\omega) h_j\bigr) +
 \frac{1}{2}
 \bigl((1-\omega) h_{j+1} \oplus \omega h_{j+2}\bigr)
 .
\]
Of course, many other subdivision schemes, including schemes of arbitrary
arity, can be constructed in this spirit. 

Let $H^\ell = S^\ell H^0 = (h^\ell_j)_{j \in \Z}$ be the sequence of Hermite 
couples $h^\ell_j = (p^\ell_j,\alpha^\ell_j)$ at level $\ell \in \N_0$.
A common feature of the algorithms presented above is the reproduction of
circles. That is, if there exists a midpoint $m \in \C$ and a radius $r>0$
such that 
$p^\ell_j = m - i r \exp(i \alpha^\ell_j)$ for all $j$, then
$p^{\ell+1}_j = m - i r \exp(i \alpha^{\ell+1}_j)$ for all $j$.
Further, clothoids are almost reproduced. That is, if $p \in \calK_1$
is a clothoid with angle function $\alpha$, and if 
\[
 p^\ell_j = p(t^\ell_j)
 ,\quad
 \alpha^\ell_j = \alpha(t^\ell_j)
 ,\quad 
 j \in \Z
 ,
\]
for certain parameters $t^\ell_j$, then there exist parameters
$t^{\ell+1}_j$ such that 
\[
 p^{\ell+1}_j \approx p(t^{\ell+1}_j)
 ,\quad
 \alpha^{\ell+1}_j \approx \alpha(t^{\ell+1}_j)
 ,\quad 
 j \in \Z
 .
\]
The quality of the approximation is determined by the 
magnitude of the angle defect according to the preceding section.

Concerning convergence, the examples in the next section suggest that all 
algorithms presented here are $G^1$-convergent in the following 
sense: There exists a curve $p$ in $\C$ with tangent angle $\alpha$ which, 
respectively, are the limits of points and angles generated by the algorithm.
More precisely, 
if $j(\ell)$ is a sequence of integers such that 
$t = \lim_{\ell \to \infty} 2^{-\ell} j(\ell)$, then 
\[
 \lim_{\ell \to \infty} p^\ell_{j(\ell)} = p(t)
 \quad \text{and} \quad
 \lim_{\ell \to \infty} \alpha^\ell_{j(\ell)} = \arg p'(t) = \alpha(t)
 .
\]
This is analogous to standard Hermite subdivision, where the 
slope of the limit must coincide with the limit of slopes,
and that is why we suggest to call the procedures introduced here
{\em Geometric Hermite subdivision}. 
A proof of $G^1$-convergence of the Lane-Riesenfeld-type algorithm $S_1$ of 
degree $1$ is given in Section~\ref{sec:Proof};
a more general theory is currently developed and beyond the scope of this 
paper.

\section{Numerical Experiments}
\label{sec:Examples}
In this section, we present numerical examples illustrating the 
shape properties of the Hermite subdivision algorithms
$S_n$ and $S^4_\omega$ as introduced in the preceding section.

Throughout, we use the same set of initial data $H^0$. To avoid a special 
treatment of boundaries, it is assumed to be periodic, 
$h^0_j = h^0_{j+8}, j \in \Z$. 
All figures are structured as follows:
On the left hand side, we see the initial points $p^0_j$ and 
normals $n^0_j := i \exp(i \alpha_j^0)$ together with the polygon 
$p^\ell_j$ as obtained after $\ell=8$ rounds of subdivision.
The middle figure shows the points $p^\ell_j$ for $\ell=5$
together with the corresponding normals $n^\ell_j := i \exp(i \alpha_j^\ell)$. 
On the right hand side, estimated curvature values 
$\kappa^\ell_j, \ell = 8,$ are plotted versus the normalized chord length 
\[
 s_j^\ell := \sigma \sum_{i=0}^{j-1} 
\bigl|p^\ell_{i+1}-p^\ell_i\bigr|
 ,
\]
where $\sigma$ is chosen such that the length of one loop equals $1$.
Specifically, the curvature value $\kappa^\ell_j$ is computed as the 
reciprocal radius 
of the circle interpolating the points $p^\ell_{j-1},p^\ell_j,p^\ell_{j+1}$.
The Fresnel-integrals $I(\alpha,\cdot)$ appearing in the 
definition of the clothoid average are computed using Gauss-Legendre 
quadrature with three nodes.

Figure~\ref{fig:gallery_1} shows the Lane-Riesenfeld-type algorithm $S_1$.
By construction, it is interpolatory. The plots suggests that the limit 
is $G^1$, i.e., free of kinks, and that the tangent angle of the limit equals 
the limit of tangent angles. The same observation is true for all 
subsequent cases. Curvature looks piecewise linear, as it would be the case 
when connecting always two consecutive points by a clothoid. In fact, the 
pieces are not exact clothoids due to the angle defect, but the deviation 
is very small. The uneven distribution of spikes in the middle figure indicates
that the standard parametrization
\[
 t^\ell_j = j 2^{-\ell} \mapsto p^\ell_j
\]
does not converge to a differentiable limit. 

Figure~\ref{fig:gallery_2} shows the Lane-Riesenfeld-type algorithm $S_2$.
The scheme is no longer interpolatory, even though the initial 
data points $p^0_j$ are very close to the generated curve. The same is true 
for all schemes $S_n, n \ge 2$.
Curvature looks continuous, but it still has certain imperfections.
Thanks to the averaging step, the standard parametrization is now smoothed out,
and we conjecture that it is $C^1$.

Figure~\ref{fig:gallery_3} shows the Lane-Riesenfeld-type algorithm $S_3$.
Now, the curvature distribution is free of artifacts so that the generated 
curve can be rated {\em Class A} according to the conventions of the automotive 
industry. However, there are certain spots where curvature seems to be not 
differentiable with respect to arc length. That is, the limit is $G^2$, but 
not $G^3$. The same seems to be true also for Lane-Riesenfeld variants of even 
higher degree.

Figure~\ref{fig:gallery_4} shows the four-point scheme with weight $\omega = 
-1/18$. It is interpolatory and seems to generate a $G^2$-limit. 
Extended experiments show that the value $\omega =-1/18$ yields visually 
fairest curves. For comparison, Figure~\ref{fig:gallery_5} shows the much less 
satisfactory result for $\omega = -1/9$.

\def\h{4.5cm}
\begin{figure}
\centering
\includegraphics[height=\h]{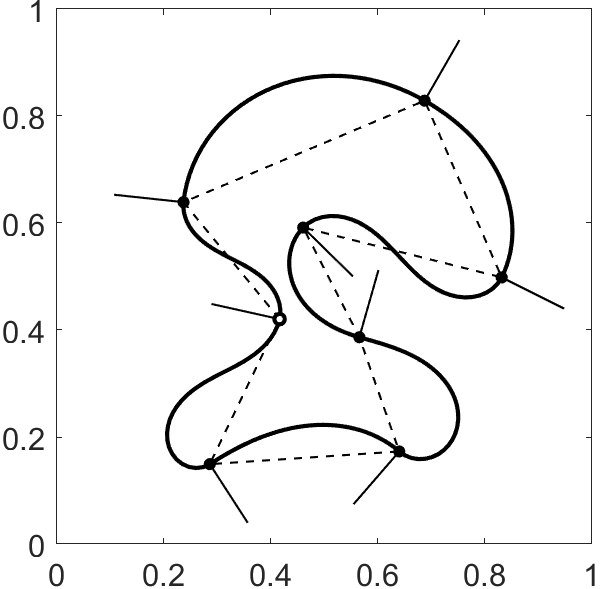} 
\quad 
\includegraphics[height=\h]{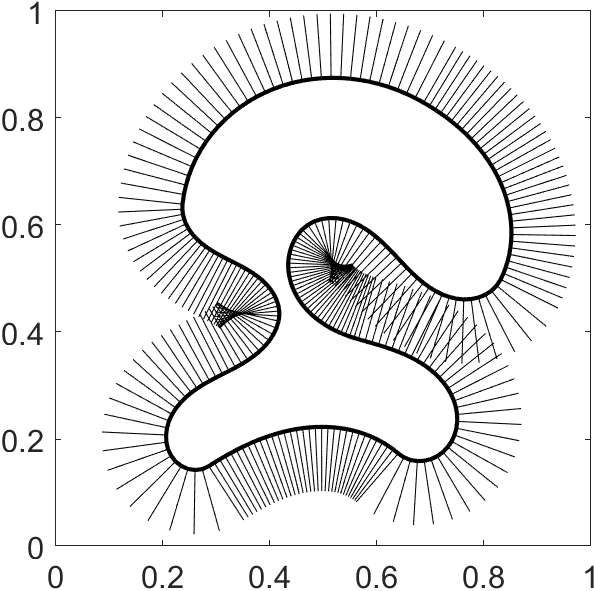}
\quad
\includegraphics[height=\h]{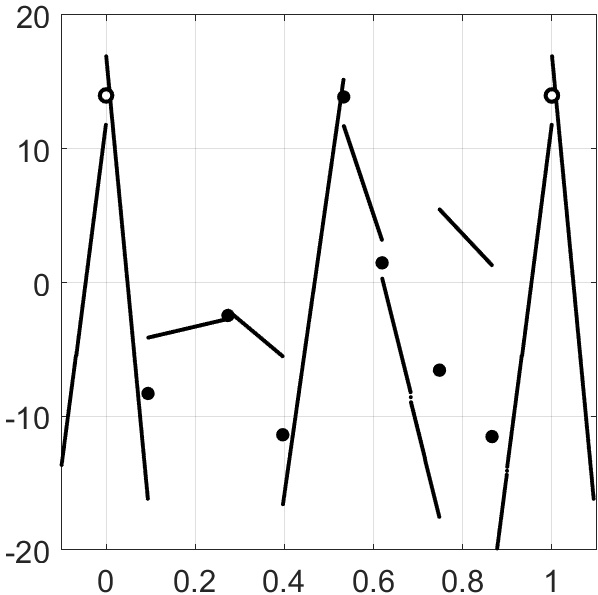}
\caption{Lane-Riesenfeld-type subdivision of degree $n=1$.}
\label{fig:gallery_1}
\end{figure}

\begin{figure}
\centering
\includegraphics[height=\h]{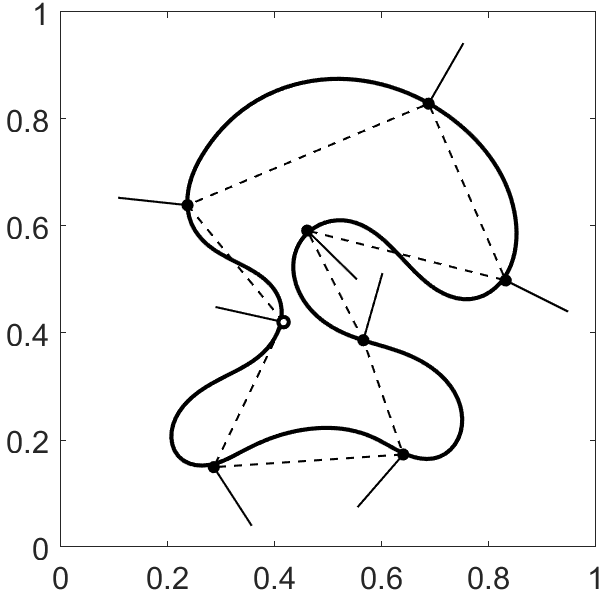} 
\quad 
\includegraphics[height=\h]{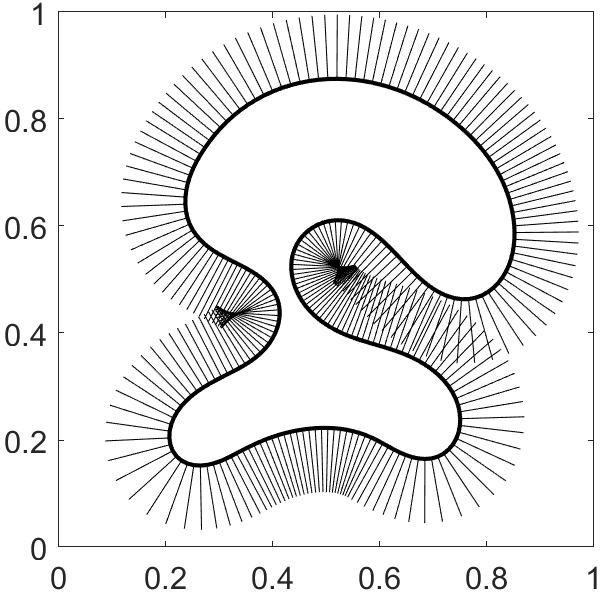}
\quad
\includegraphics[height=\h]{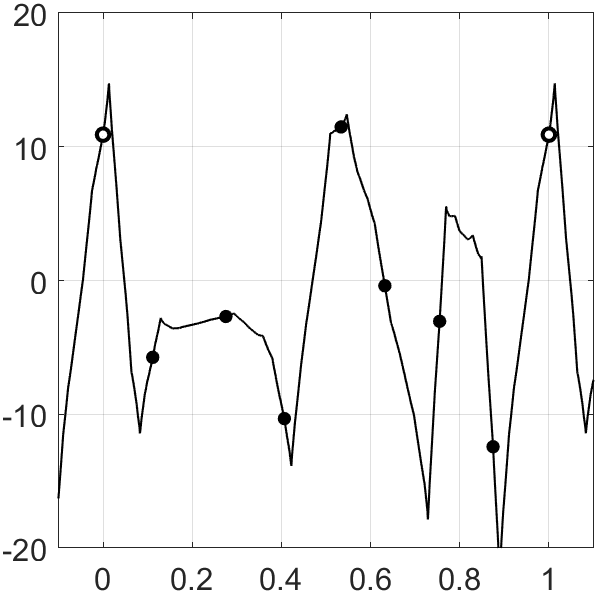}
\caption{Lane-Riesenfeld-type subdivision of degree $n=2$.}
\label{fig:gallery_2}
\end{figure}

\begin{figure}
\centering
\includegraphics[height=\h]{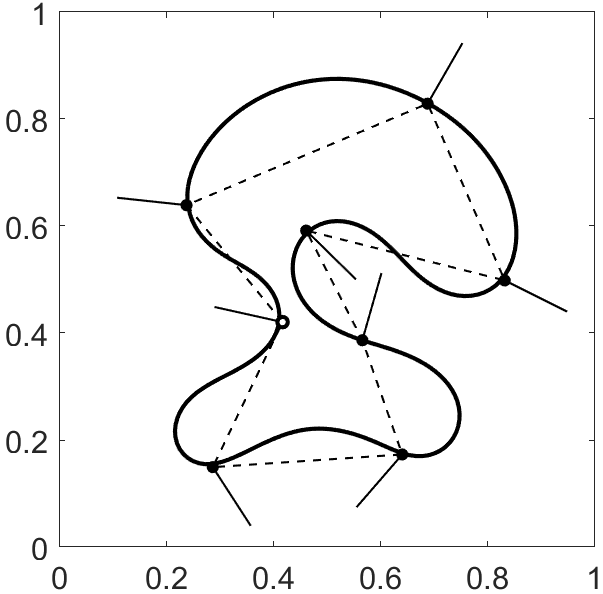} 
\quad 
\includegraphics[height=\h]{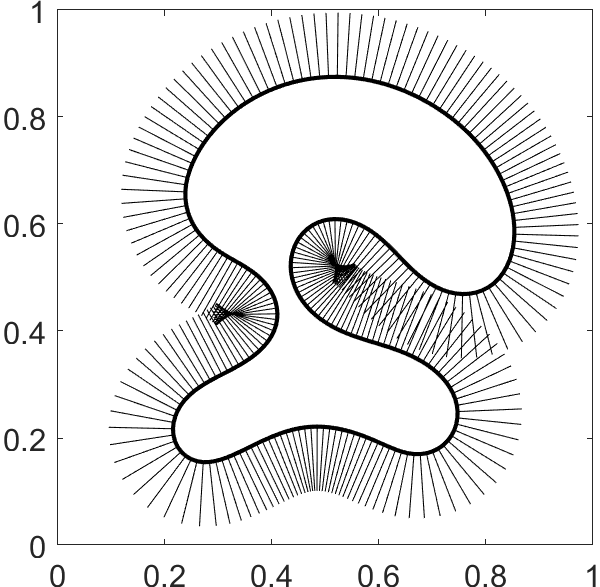}
\quad
\includegraphics[height=\h]{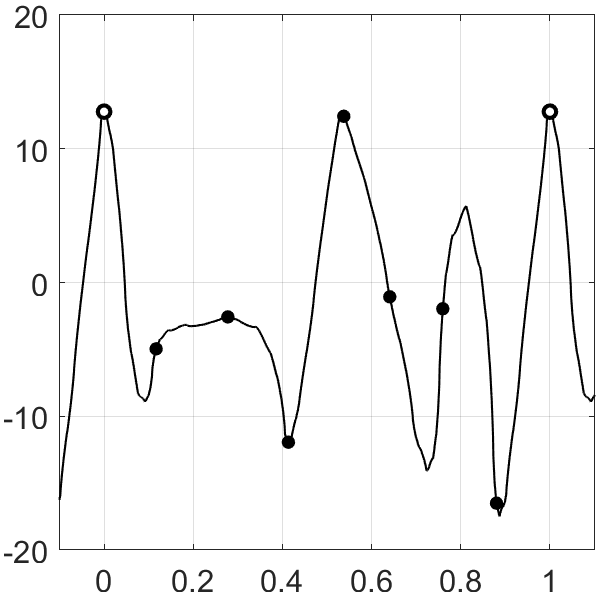}
\caption{Lane-Riesenfeld-type subdivision of degree $n=3$.}
\label{fig:gallery_3}
\end{figure}

\begin{figure}
\centering
\includegraphics[height=\h]{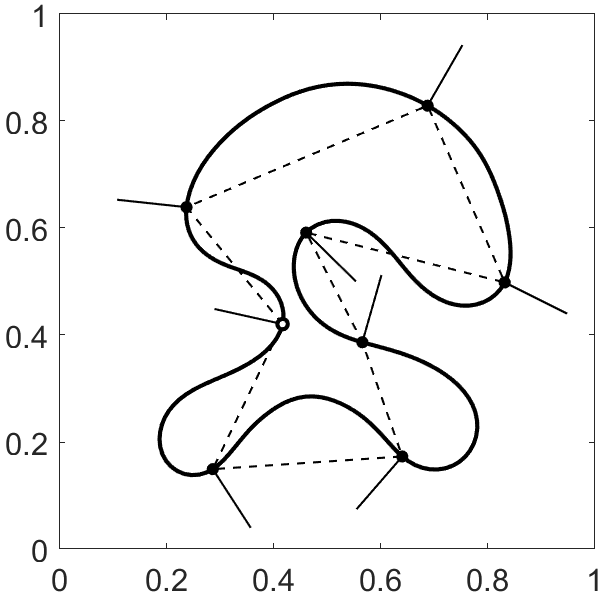} 
\quad 
\includegraphics[height=\h]{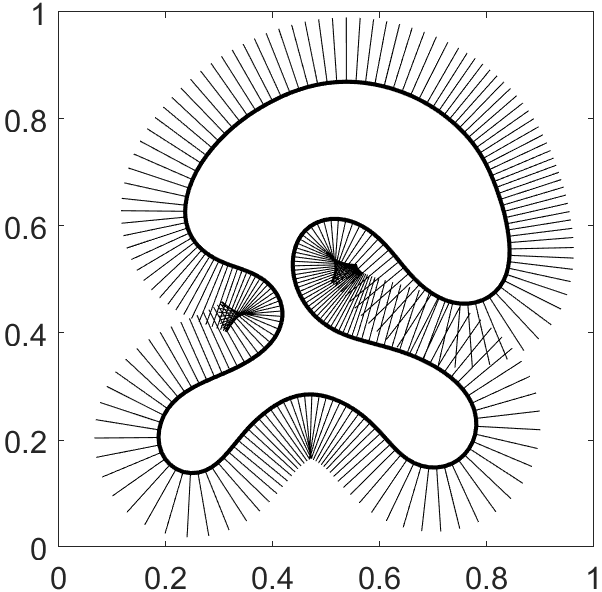}
\quad
\includegraphics[height=\h]{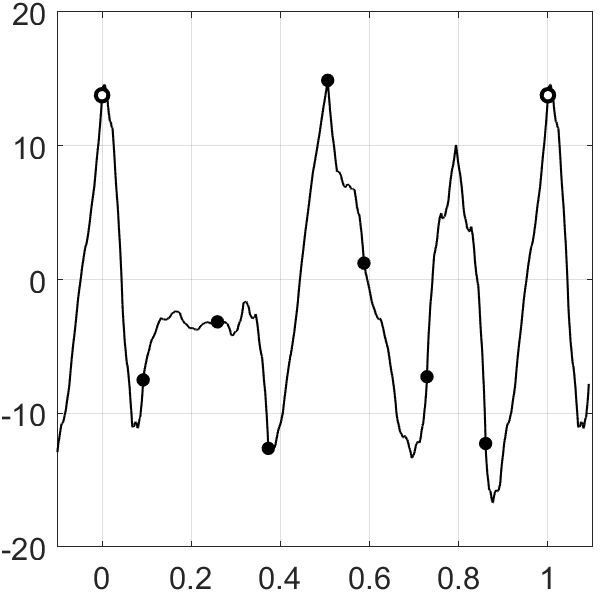}
\caption{Four-point subdivision with weight $\omega=-1/18$.}
\label{fig:gallery_4}
\end{figure}

\begin{figure}
\centering
\includegraphics[height=\h]{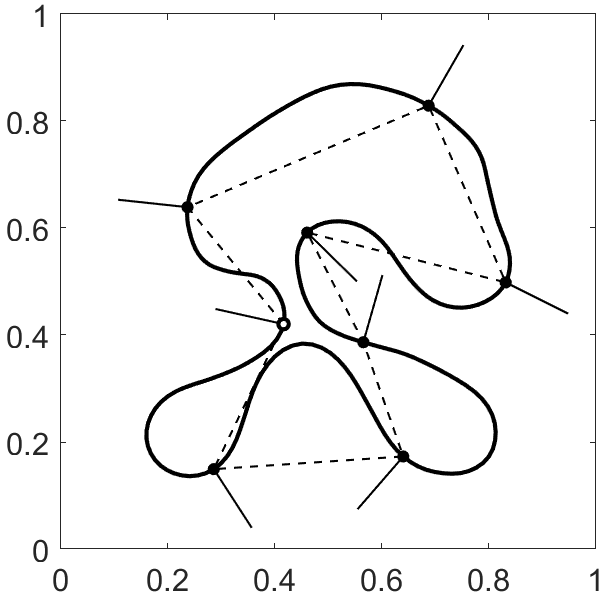} 
\quad 
\includegraphics[height=\h]{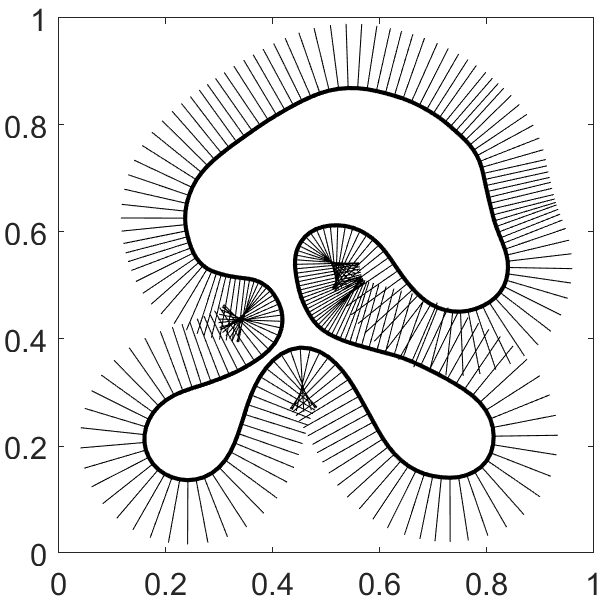}
\quad
\includegraphics[height=\h]{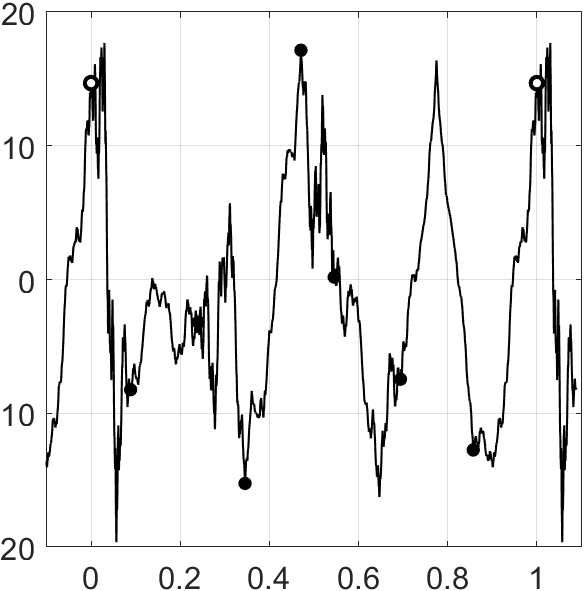}
\caption{Four-point subdivision with weight $\omega=-1/9$.}
\label{fig:gallery_5}
\end{figure}


\section{$G^1$-convergence of the scheme $S_1$}
\label{sec:Proof}

While the focus of this paper is on the construction of geometric Hermite 
subdivision schemes, we also want to outline a proof for the $G^1$-convergence 
of the Lane-Riesenfeld-type algorithm $S_1$. It is based on the results in 
\cite{dyn2012geometric}, but we want to remark that the notion of $G^1$ used 
there is special. It would be good to have a link between the approach of 
Dyn/Hormann and standard theory, which calls a curve $G^1$ if it has a 
$C^1$-reparametrization. 

Convergence and smoothness are local properties of $S_1$ in the 
sense that the Hermite couples $h^\ell_j$ depend only on $h^0_k,h^0_{k+1}$
for $k \le j2^{-\ell} \le k+1$, corresponding to the interval $[k,k+1]$ of
the standard parametrization. That is, convergence and smoothness
can be studied for initial data $H^0$ which are periodic 
in the sense that $h^0_j = h^0_{j+k}, j \in \Z$, for some $k \ge 2$.
This assumption avoids technical problems with unboundedness of sequences.
Throughout, as before, $H^\ell = (h^\ell_j)_{j \in \Z}, h^\ell_j = 
(p^\ell_j,\alpha^\ell_j)$. Further, we define the vectors
$d^\ell_j := p^\ell_{j+1}-p^\ell_j$ and  the pairs of angles 
$\beta^\ell_{j,J} = (\beta^\ell_{j,0},\beta^\ell_{j,1})$ by 
\[
 \beta^\ell_{j,0} := \alpha^\ell_j - \arg d^\ell_j
 ,\quad 
 \beta^\ell_{j,1} := \alpha^\ell_{j+1} - \arg d^\ell_j
 .
\]
The Euclidean disk in $\R^2$ with radius $3\pi/4$ is 
denoted by 
$B^* := \{\beta_J : \|\beta_J\|_2 \le 3\pi/4\}$.

\begin{theorem}
Let $H^0$ be periodic. If $d^0_j \neq 0$ and 
$\beta^0_{j,J} \in B^*$ for all $j \in \Z$, then
the iterates $H^\ell := S_1^\ell H^0$ define a sequence of polygons 
$(p^\ell_j)_{j \in \Z}$ converging to a $G^1$-limit in the sense of 
\cite{dyn2012geometric}. Moreover, the angles $\alpha^\ell_j$ converge
to the arguments of $d^\ell_j$,
\[
     \lim_{\ell \to \infty} \max_{j \in \Z} 
     \bigl(\alpha^\ell_j - \arg(d^\ell_j)\bigr)  = 0
     .
\]
\end{theorem}
\begin{proof}
We consider the two-point problem in normal position,
$h_0 = (0,\beta_0), h_1 = (1,\beta_1)$, with secant $d = 1 - 0 = 1$
and angles $(\beta_0,\beta_1) \in B^*$. The clothoid average
\[
 h_{1/2} = (p_{1/2},\alpha_{1/2}) 
 := \frac{1}{2} h_0 \oplus \frac{1}{2} h_1
\]
yields new secants 
$d'_0 = p_{1/2}, d'_1 = 1 - p_{1/2}$ 
and new pairs of angles 
\[
 \beta'_{0,J} = (\beta_0 - \arg d'_0,\alpha_{1/2} - \arg d'_0)
 ,\quad 
 \beta'_{1,J} = (\alpha_{1/2} - \arg d'_1,\beta_1 - \arg d'_1)
 .
\]
Figure~\ref{fig:contraction} shows that the ratios
\[
 r(\beta) := \frac{\max\bigl\{|d'_0|,|d'_1|\bigr\}}{|d|}
 ,\quad 
 \rho(\beta) := \frac{\max\bigl\{\|\beta'_{0,J}\|_2,\|\beta'_{1,J}\|_2\bigr\}}
 {\|\beta_J\|_2}
 ,\quad 
 \beta \in B^*
 ,
\]
are bounded by $\|r\|_\infty \le 4/5$ and $\|\rho\|_\infty \le 19/20$, 
respectively.
Now, consider a pair of consecutive Hermite couples 
$h^{\ell-1}_j,h^{\ell-1}_{j+1}$
and its two descendants. The ratios
\[
 r^\ell_j := \frac{\max\bigl\{|d^{\ell}_{2j}|, 
 |d^{\ell}_{2j+1}|\bigr\}}{|d^{\ell-1}_j|}
 ,\quad 
 \rho^\ell_j := 
\frac{\max\bigl\{\|\beta^{\ell}_{2j,J}\|_2,\|\beta^{\ell}_{2j+1,J}\|_2\bigr\}}
 {\|\beta^{\ell-1}_{j,J}\|_2}
\]
are invariant under similarities so that we may move the configuration
to normal position, showing that $r^\ell_j \le \|r\|_\infty \le 4/5$ and
$\rho^\ell_j \le \|\rho\|_\infty \le 19/20$. The latter bound guarantees that 
$\beta^\ell_{j,J} \in B^*$ for all $j \in \Z$ and $\ell \in \N$.
Taking the maximum over all 
elements at level $\ell$ and iterating backwards, we 
obtain 
\[
  \max_j |d^{\ell}_j| \le (4/5)^\ell \max_j |d^0_j|
  ,\quad 
  \max_j \|\beta^{\ell}_{j,J}\|_2 \le (19/20)^\ell \max_j \|\beta^0_{j,J}\|_2
  .
\]
First, the sequence of maximal secant lengths is summable, implying that 
the sequence of polygons $(p^\ell_j)_{j \in \Z}$, converges to a 
continuous limit, see Theorem~3 and Proposition~4 in \cite{dyn2012geometric}.
Second,
to address $G^1$-continuity, we consider the exterior angles 
\[
  \delta^\ell_j := \arg(d^\ell_j) - \arg(d^\ell_{j-1}) = 
  \beta^\ell_{j-1,1} - \beta^\ell_{j,0}
\]
between consecutive secants. They are bounded by 
\[
 \max_{j \in \Z} |\delta^\ell_j| \le 2 \max_{j \in \Z} \|\beta^\ell_{j,J}\|_2
 \le 
 2 \cdot (19/20)^\ell \max_{j \in \Z} \|\beta^0_{j,J}\|_2
 .
\]
Hence, also the sequence of maximal exterior angles is summable, implying that 
the limit curve is $G^1$, see Theorem~18 in \cite{dyn2012geometric}.
Third, the final statement of the theorem is a simple consequence of 
$\alpha^\ell_j - \arg(d^\ell_j) = \beta^\ell_{j,0}$.

\begin{figure}
 \centering
 \includegraphics[height=5cm]{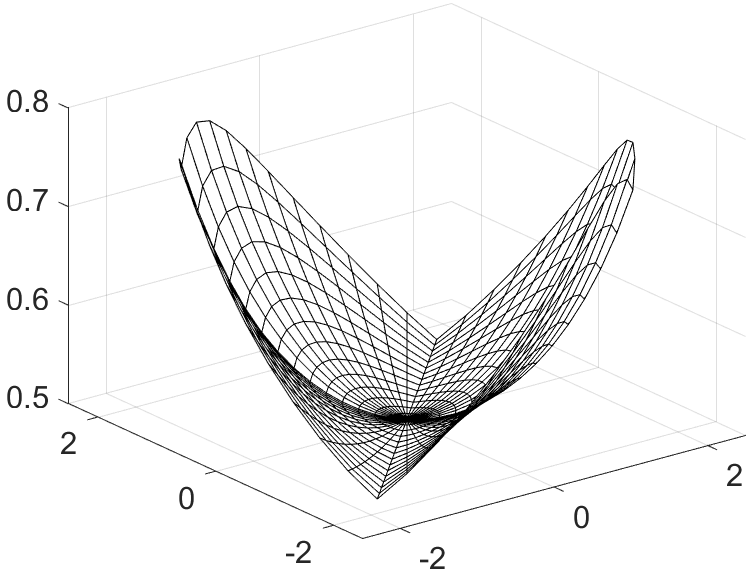}
 \qquad 
 \includegraphics[height=5cm]{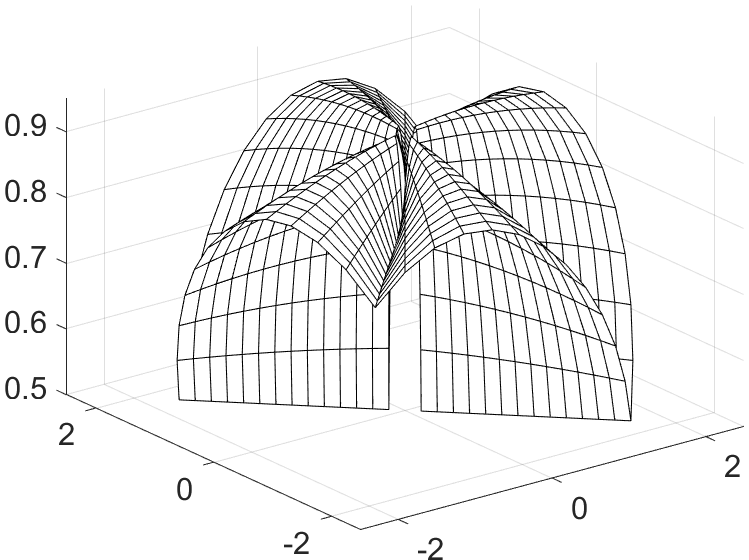}
 \caption{Contraction rates for lengths {\em (left)} 
 and angles {\em (right)}.}
 \label{fig:contraction}
\end{figure}

\end{proof}
\section{Conclusion and Outlook}
\label{sec:Conclusion}

In this paper, we have proposed geometric Hermite subdivision schemes and we 
have demonstrated that they are a reasonable means for designing curves with 
prescribed tangents or normals.
More precisely, we have first proposed an explicit strategy to 
approximate Hermite interpolating clothoids and used it to define the
clothoid averages.
Then we have used clothoidal averaging to define geometric Hermite subdivision 
schemes.
Particular instances considered 
were the geometric Hermite analogues of the 
Lane-Riesenfeld schemes and of the four-point scheme.
Examples demonstrate that these schemes yield very convincing results.
Finally, we have presented some first smoothness results. More precisely,
we obtain smoothness in the sense of \cite{dyn2012geometric} for the first order 
geometric Lane-Riesenfeld Hermite scheme. However, the notion of 
\cite{dyn2012geometric} is not standard; 
standard theory calls a curve $G^1$ if it has a $C^1$ reparametrization.
It would be desirable to obtain related $G^1$- or even 
$G^2$-smoothness results for wider classes of schemes.
This is an interesting topic of ongoing research.

\bibliographystyle{alpha}
\bibliography{CAandHS}

\newcommand{\etalchar}[1]{$^{#1}$}
\begin{thebibliography}{HEWF13}

\bibitem[AEV03]{aspert2003non}
Nicolas Aspert, Touradj Ebrahimi, and Pierre Vandergheynst.
\newblock Non-linear subdivision using local spherical coordinates.
\newblock {\em Computer Aided Geometric Design}, 20(3):165--187, 2003.

\bibitem[AGRA15]{alia2015local}
Chebly Alia, Tagne Gilles, Talj Reine, and Charara Ali.
\newblock Local trajectory planning and tracking of autonomous vehicles, using
  clothoid tentacles method.
\newblock In {\em 2015 IEEE Intelligent Vehicles Symposium (IV)}, pages
  674--679. IEEE, 2015.

\bibitem[Baa84]{baass1984use}
K.G. Baass.
\newblock The use of clothoid templates in highway design.
\newblock {\em Transportation Forum}, 1:47--52, 1984.

\bibitem[BdB65]{birkhoff1964piecewise}
Garrett Birkhoff and Carl de~Boor.
\newblock Nonlinear splines.
\newblock In {\em Proc. General Motors Symp. of 1964}, page 164–190. 1965.

\bibitem[Ber15]{berkemeier2015clothoid}
Matthew Berkemeier.
\newblock Clothoid segments for optimal switching between arcs during low-speed
  {A}ckerman path tracking with rate-limited steering.
\newblock In {\em 2015 American Control Conference (ACC)}, pages 501--506.
  IEEE, 2015.

\bibitem[BF15]{bertolazzi2015g1}
Enrico Bertolazzi and Marco Frego.
\newblock {$G^1$} fitting with clothoids.
\newblock {\em Mathematical Methods in the Applied Sciences}, 38(5):881--897,
  2015.

\bibitem[BF18]{bertolazzi2018interpolating}
Enrico Bertolazzi and Marco Frego.
\newblock Interpolating clothoid splines with curvature continuity.
\newblock {\em Mathematical Methods in the Applied Sciences}, 41(4):1723--1737,
  2018.

\bibitem[BL15]{lun2015inpainting}
Budianto and Daniel Lun.
\newblock Inpainting for fringe projection profilometry based on geometrically
  guided iterative regularization.
\newblock {\em IEEE Transactions on Image Processing}, 24(12):5531--5542, 2015.

\bibitem[BLP10]{baran2010sketching}
Ilya Baran, Jaakko Lehtinen, and Jovan Popovi{\'c}.
\newblock Sketching clothoid splines using shortest paths.
\newblock In {\em Computer Graphics Forum}, volume~29, pages 655--664. Wiley
  Online Library, 2010.

\bibitem[BLR{\etalchar{+}}12]{biral2012intersection}
Francesco Biral, Roberto Lot, Stefano Rota, Marco Fontana, and V{\'e}ronique
  Huth.
\newblock Intersection support system for powered two-wheeled vehicles:
  {T}hreat assessment based on a receding horizon approach.
\newblock {\em IEEE Transactions on Intelligent Transportation Systems},
  13(2):805--816, 2012.

\bibitem[CDM91]{cavaretta1991stationary}
Alfred Cavaretta, Wolfgang Dahmen, and Charles Micchelli.
\newblock {\em Stationary subdivision}, volume 453.
\newblock American Mathematical Soc., 1991.

\bibitem[CDM03]{cohen2003quasilinear}
Albert Cohen, Nira Dyn, and Basarab Matei.
\newblock Quasilinear subdivision schemes with applications to {ENO}
  interpolation.
\newblock {\em Applied and Computational Harmonic Analysis}, 15(2):89--116,
  2003.

\bibitem[CHR13]{cashman2013generalized}
Thomas Cashman, Kai Hormann, and Ulrich Reif.
\newblock Generalized {L}ane--{R}iesenfeld algorithms.
\newblock {\em Computer Aided Geometric Design}, 30(4):398--409, 2013.

\bibitem[CJ07]{chalmoviansky2007non}
Pavel Chalmoviansky and Bert J{\"u}ttler.
\newblock A non-linear circle-preserving subdivision scheme.
\newblock {\em Advances in Computational Mathematics}, 27(4):375--400, 2007.

\bibitem[Coo93]{coope1993curve}
Ian Coope.
\newblock Curve interpolation with nonlinear spiral splines.
\newblock {\em IMA Journal of Numerical Analysis}, 13(3):327--341, 1993.

\bibitem[DFH09]{dyn2009four}
Nira Dyn, Michael Floater, and Kai Hormann.
\newblock Four-point curve subdivision based on iterated chordal and
  centripetal parameterizations.
\newblock {\em Computer Aided Geometric Design}, 26(3):279--286, 2009.

\bibitem[DH12]{dyn2012geometric}
Nira Dyn and Kai Hormann.
\newblock Geometric conditions for tangent continuity of interpolatory planar
  subdivision curves.
\newblock {\em Computer Aided Geometric Design}, 29(6):332--347, 2012.

\bibitem[DL02]{dyn2002subdivision}
Nira Dyn and David Levin.
\newblock Subdivision schemes in geometric modelling.
\newblock {\em Acta Numerica}, 11:73--144, 2002.

\bibitem[DY00]{donoho2000nonlinear}
David Donoho and Thomas Yu.
\newblock Nonlinear pyramid transforms based on median-interpolation.
\newblock {\em SIAM Journal on Mathematical Analysis}, 31(5):1030--1061, 2000.

\bibitem[ERS15]{ewald2015holder}
Tobias Ewald, Ulrich Reif, and Malcolm Sabin.
\newblock H{\"o}lder regularity of geometric subdivision schemes.
\newblock {\em Constructive Approximation}, 42(3):425--458, 2015.

\bibitem[Gro10]{grohs2010general}
Philipp Grohs.
\newblock A general proximity analysis of nonlinear subdivision schemes.
\newblock {\em SIAM Journal on Mathematical Analysis}, 42(2):729--750, 2010.

\bibitem[GVS09]{Goldman2009OnTS}
Ron Goldman, Etienne Vouga, and Scott Schaefer.
\newblock On the smoothness of real-valued functions generated by subdivision
  schemes using nonlinear binary averaging.
\newblock {\em Computer Aided Geometric Design}, 26:231--242, 2009.

\bibitem[Han18]{han2018framelets}
Bin Han.
\newblock {\em Framelets and wavelets: Algorithms, analysis, and applications}.
\newblock Springer, 2018.

\bibitem[HEWF13]{havemann2013curvature}
Sven Havemann, Johannes Edelsbrunner, Philipp Wagner, and Dieter Fellner.
\newblock Curvature-controlled curve editing using piecewise clothoid curves.
\newblock {\em Computers \& graphics}, 37(6):764--773, 2013.

\bibitem[HL92]{hoschek1992grundlagen}
Josef Hoschek and Dieter Lasser.
\newblock {\em Grundlagen der geometrischen {D}atenverarbeitung}.
\newblock Teubner, 1992.

\bibitem[KFP03]{kimia2003euler}
Benjamin Kimia, Ilana Frankel, and Ana-Maria Popescu.
\newblock Euler spiral for shape completion.
\newblock {\em International Journal of Computer Vision}, 54(1-3):159--182,
  2003.

\bibitem[KVD02]{kuijt2002shape}
Frans Kuijt and Ruud Van~Damme.
\newblock Shape preserving interpolatory subdivision schemes for nonuniform
  data.
\newblock {\em Journal of Approximation Theory}, 114(1):1--32, 2002.

\bibitem[LD16]{lipovetsky2016weighted}
Evgeny Lipovetsky and Nira Dyn.
\newblock A weighted binary average of point-normal pairs with application to
  subdivision schemes.
\newblock {\em Computer Aided Geometric Design}, 48:36--48, 2016.

\bibitem[MDL05]{marinov2005geometrically}
Martin Marinov, Nira Dyn, and David Levin.
\newblock Geometrically controlled 4-point interpolatory schemes.
\newblock In {\em Advances in multiresolution for geometric modelling}, pages
  301--315. Springer, 2005.

\bibitem[Meh74]{mehlum1974nonlinear}
Even Mehlum.
\newblock Nonlinear splines.
\newblock In {\em Computer Aided Geometric Design}, pages 173--207. Elsevier,
  1974.

\bibitem[Moo16]{moosm}
Caroline Moosm{\"u}ller.
\newblock {$C^{1}$} {A}nalysis of {H}ermite subdivision schemes on manifolds.
\newblock {\em SIAM Journal on Numerical Analysis}, 54(5):3003--3031, 2016.

\bibitem[MS09]{mccrae2009sketching}
James McCrae and Karan Singh.
\newblock Sketching piecewise clothoid curves.
\newblock {\em Computers \& Graphics}, 33(4):452--461, 2009.

\bibitem[MW90]{meek1990offset}
Dereck Meek and Desmond Walton.
\newblock Offset curves of clothoidal splines.
\newblock {\em Computer-Aided Design}, 22(4):199--201, 1990.

\bibitem[MW92]{meek1992clothoid}
Dereck Meek and Desmond Walton.
\newblock Clothoid spline transition spirals.
\newblock {\em Mathematics of Computation}, 59(199):117--133, 1992.

\bibitem[MW04]{meek2004arc}
Dereck Meek and Desmond Walton.
\newblock An arc spline approximation to a clothoid.
\newblock {\em Journal of Computational and Applied Mathematics},
  170(1):59--77, 2004.

\bibitem[MW09]{meek2009two}
Dereck Meek and Desmond Walton.
\newblock A two-point ${G}^1$ {H}ermite interpolating family of spirals.
\newblock {\em Journal of Computational and Applied Mathematics},
  223(1):97--113, 2009.

\bibitem[PR08]{peters2008subdivision}
J{\"o}rg Peters and Ulrich Reif.
\newblock {\em Subdivision surfaces}.
\newblock Springer, 2008.

\bibitem[SD04]{sabin2004circle}
Malcolm Sabin and Neil Dodgson.
\newblock A circle-preserving variant of the four-point subdivision scheme.
\newblock {\em Mathematical Methods for Curves and Surfaces: Troms{\o}}, pages
  275--286, 2004.

\bibitem[SK00]{schneider2000discrete}
Robert Schneider and Leif Kobbelt.
\newblock Discrete fairing of curves and surfaces based on linear curvature
  distribution.
\newblock In {\em Curve and Surface Design, Saint-Malo 1999, Innovations in
  Applied Mathematics}, pages 371--380, 2000.

\bibitem[Sto82]{stoer1982curve}
Josef Stoer.
\newblock Curve fitting with clothoidal splines.
\newblock {\em J. Res. Nat. Bur. Standards}, 87(4):317--346, 1982.

\bibitem[WD05]{wallner2005convergence}
Johannes Wallner and Nira Dyn.
\newblock Convergence and ${C}^1$ analysis of subdivision schemes on manifolds
  by proximity.
\newblock {\em Computer Aided Geometric Design}, 22(7):593--622, 2005.

\bibitem[Wei10]{weinmann2010nonlinear}
Andreas Weinmann.
\newblock Nonlinear subdivision schemes on irregular meshes.
\newblock {\em Constructive Approximation}, 31(3):395--415, 2010.

\bibitem[WM09]{WALTON200986}
Desmond Walton and Dereck Meek.
\newblock ${G}^1$ interpolation with a single {C}ornu spiral segment.
\newblock {\em Journal of Computational and Applied Mathematics}, 223(1):86 --
  96, 2009.

\bibitem[XY09]{xie2009smoothness}
Gang Xie and Thomas Yu.
\newblock Smoothness equivalence properties of general manifold-valued data
  subdivision schemes.
\newblock {\em Multiscale Modeling \& Simulation}, 7(3):1073--1100, 2009.

\end{thebibliography}


\end{document}